\RequirePackage{easybmat}

\documentclass[11pt]{amsart}
\usepackage{etex}
\usepackage{amsmath, amsthm}%
\usepackage{amsfonts}%
\usepackage{amssymb}%
\usepackage{graphicx}
\usepackage{amscd, amstext}
\usepackage{hyperref}
\usepackage{pictexwd}
\usepackage{dcpic}
\usepackage[mathscr]{eucal}
\usepackage{multicol}
\usepackage{color}
\theoremstyle{definition}
\newtheorem{theorem}{Theorem}[section]
\newtheorem{proposition}[theorem]{Proposition}
\newtheorem{remark}[theorem]{Remark}
\newtheorem{example}[theorem]{Example}
\newtheorem{definition}[theorem]{Definition}

\newtheorem{corollary}[theorem]{Corollary}
\newtheorem{lemma}[theorem]{Lemma}
\newtheorem{question}[theorem]{Question}
\newtheorem{algorithm}[theorem]{Algorithm}

\makeatletter

\newenvironment{figurehere}
{\def\@captype{figure}}
{}
\makeatother

\begin{document}

\title{Short Tops and Semistable Degenerations}

\author[Davis]{Ryan Davis}
\address{Department of Mathematics\\
Hibbard Humanities Hall 508\\
University of Wisconsin-Eau Claire\\
Eau Claire WI 54702\\
United States of America}
\curraddr{Center for Nanohybrid Functional Materials\\
University of Nebraska-Lincoln\\
239N Scott Engineering Center\\
844 N. 16th Street\\
Lincoln, NE 68588\\
United States of America}
\email{ryan.stance.davis@gmail.com}

\author[Doran]{Charles Doran}
\address{Department of Mathematical and Statistical Science\\
632 CAB\\
University of Alberta\\
Edmonton, Alberta  T6G 2G1\\
Canada}
\email{charles.doran@ualberta.ca}

\author[Gewiss]{Adam Gewiss}
\address{Department of Mathematics\\
Hibbard Humanities Hall 508\\
University of Wisconsin-Eau Claire\\
Eau Claire WI 54702\\
United States of America}
\email{gewissa@uwec.edu}

\author[Novoseltsev]{Andrey Novoseltsev}
\address{Department of Mathematical and Statistical Science\\
632 CAB\\
University of Alberta\\
Edmonton, Alberta  T6G 2G1\\
Canada}
\email{novoselt@ualberta.ca}

\author[Skjorshammer]{Dmitri Skjorshammer}
\address{Harvey Mudd College\\
Department of Mathematics \\
301 Platt Boulevard\\
Claremont, CA 91711\\
United States of America}
\email{dmitriskj@gmail.com}

\author[Syryczuk]{Alexa Syryczuk}
\address{Department of Mathematics\\
Hibbard Humanities Hall 508\\
University of Wisconsin-Eau Claire\\
Eau Claire WI 54702\\
United States of America}
\email{syryczar@uwec.edu}

\author[Whitcher]{Ursula Whitcher}
\address{Department of Mathematics\\
Hibbard Humanities Hall 508\\
University of Wisconsin-Eau Claire\\
Eau Claire WI 54702\\
United States of America}
\email{whitchua@uwec.edu}

\thanks{We thank the anonymous referee for insightful comments, and the Blugold Commitment Differential Tuition grants program at the University of Wisconsin--Eau Claire for ongoing support of this work.  The fifth and seventh authors were supported in part by NSF Grant No. DMS-0821725.}


\date{}

\begin{abstract}One may construct a large class of Calabi-Yau varieties by taking
anticanonical hypersurfaces in toric varieties obtained from reflexive
polytopes. If the intersection of a reflexive polytope with a hyperplane
through the origin yields a lower-dimensional reflexive polytope, then the
corresponding Calabi-Yau varieties are fibered by lower-dimensional Calabi-Yau
varieties. A top generalizes the idea of splitting a reflexive polytope into
two pieces. In contrast to the classification of reflexive polytopes, there are
infinite families of equivalence classes of tops. Tops may be used to describe
either fibrations or degenerations of Calabi-Yau varieties. We give a simple
combinatorial condition on tops which produces semistable degenerations of K3
surfaces, and, when appropriate smoothness conditions are met, semistable
degenerations of Calabi-Yau threefolds. Our method is constructive: given a
fixed reflexive polytope which will lie on the boundary of the top, we describe
an algorithm for constructing tops which yields semistable degenerations of the
corresponding hypersurfaces. The properties of each degeneration may be
computed directly from the combinatorial structure of the top.
\end{abstract}

\maketitle

\section{Introduction}

\subsection{The combinatorics of tops}

Let $N \cong \mathbb{Z}^k$ be a lattice, with dual lattice $M$.  
\textcolor{black}{We write points in $N$ as $(n_1, \dots, n_k)$, and points in $M$ as $(m_1, \dots, m_k)$.  We write points in the associated real vector spaces $N_\mathbb{R}$ and $M_\mathbb{R}$ as $\vec{x} = (x_1, \dots, x_k)$ and $\vec{y} = (y_1, \dots, y_k)$, respectively.}  

\textcolor{black}{We will need to use the data of a \emph{triangulation} of a lattice polytope.}

\textcolor{black}{\begin{definition}Let $\mathcal{A}$ be a finite set of points in a $d$-dimensional real vector space.  A \emph{triangulation} of $\mathcal{A}$ is a collection $\mathcal{T}$ of $d$-dimensional simplices such that the vertices of the simplices are points in $\mathcal{A}$, the union of the simplices is the convex hull of $\mathcal{A}$, and any pair of simplices in the collection intersects in a common (possibly empty) face.  We refer to the faces of the simplices in $\mathcal{T}$ as \emph{faces of the triangulation} $\mathcal{T}$.  If $\Delta$ is a lattice polytope, then we refer to a triangulation of the set of lattice points in $\Delta$ as a \emph{triangulation of the lattice polytope} $\Delta$.
\end{definition}}

\textcolor{black}{We may use the duality between $N$ and $M$ to construct new polytopes.}

\begin{definition}Let $\Delta$ be a lattice polytope in \textcolor{black}{$N_\mathbb{R}$} which contains $\vec{0}$.  The \emph{polar polytope} $\Delta^\circ$ is the polytope in \textcolor{black}{$M_\mathbb{R}$} given by:

\begin{align*}
\{(y_1, \dots, y_k) : (n_1, \dots, n_k) \cdot (y_1, \dots, y_k) \geq -1 \\
 \mathrm{for}\;\mathrm{all}\;(n_1,\dots,n_k) \in \Delta\}
\end{align*}
\end{definition}

\begin{definition}
We say a lattice polytope $\Diamond$ is \emph{reflexive} if its polar polytope $\Diamond^\circ$ is also a lattice polytope.
\end{definition}

One may construct a large class of Calabi-Yau varieties by taking anticanonical hypersurfaces in toric varieties obtained from reflexive polytopes.  The polar duality relationship between pairs of reflexive polytopes induces the mirror relationship on the corresponding Calabi-Yau varieties.  To understand the Calabi-Yau varieties arising in this fashion, one must classify the corresponding reflexive polytopes.  There is $1$ one-dimensional reflexive polytope, and there are $16$ isomorphism classes of two-dimensional reflexive polytopes.  The physicists Kreuzer and Skarke showed that there are ${4,319}$ classes of three-dimensional reflexive polytopes and ${473,800,776}$ classes of four-dimensional reflexive polytopes.  

In dimensions $5$ and above, the classification of reflexive polytopes is an open problem.  \textcolor{black}{However, an algorithm for constructing and classifying a restricted class of reflexive polytopes called \emph{smooth Fano polytopes} was given by \cite{Obro}.  We say a reflexive polytope is a smooth Fano polytope if the vertices of every facet of the polytope form a $\mathbb{Z}$-basis for $N$.  Note that every facet of a smooth Fano polytope has exactly $k$ vertices.}

If the intersection of a reflexive polytope with a hyperplane through the origin yields a lower-dimensional reflexive polytope, then the corresponding Calabi-Yau varieties are fibered by lower-dimensional Calabi-Yau varieties.  This relationship has been studied extensively in the physics literature, and more recently by Grassi and Perduca in \cite{GP}.  The physicists Candelas and Font generalized the concept of a reflexive polytope sliced by a hyperplane to the idea of a \emph{top} in \cite{CF}.  

\begin{definition}
A $k$-dimensional \emph{top} is a lattice polytope \textcolor{black}{in $N$ which} has one facet which contains the origin and consists of a $k-1$-dimensional reflexive polytope, \textcolor{black}{and where the} other facets of the \textcolor{black}{polytope} are given by equations of the form 
\[(x_1, \dots, x_k) \cdot (n_1, \dots, n_k) = -1.\]
\noindent Here, $(n_1, \dots, n_k)$ is a point in the lattice $N$.  \textcolor{black}{We refer to the facet containing the origin as the \emph{reflexive boundary}.}  
\end{definition}

By applying an appropriate change of coordinates, we may assume that the reflexive boundary corresponds to the points of the top satisfying
\[(n_1,\dots,n_k) \cdot (0,\dots,0,1) = 0.\]
In this case, the reflexive boundary is simply the intersection of $\Diamond$ with the hyperplane $x_k = 0$.  We choose the convention that all tops are contained in the half-space $x_k \geq 0$.

A two-dimensional top is shown in Figure~\ref{F:top}.  

\begin{figure}[h!]
\begin{center}
\scalebox{.9}{\includegraphics{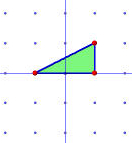}}
\end{center}
\caption{A two-dimensional top}
\label{F:top}
\end{figure}

The polar dual $\Diamond^\circ$ of $\Diamond$ is an unbounded lattice polyhedron; in our choice of coordinates, we may assume that $\Diamond^\circ$ extends infinitely in the $y_k$ direction.  A dual top is illustrated in Figure~\ref{F:dualtop}.  The projection map $(y_1, \dots, y_k) \mapsto (y_1, \dots, y_{k-1})$ maps $\Diamond^\circ$ onto a $k-1$-dimensional reflexive polytope which we call the \emph{dual reflexive boundary}.  As the name implies, the dual reflexive boundary of a dual top is the polar dual of the reflexive boundary of the corresponding top.

\begin{figure}[h!]
\begin{center}
\scalebox{.8}{\includegraphics{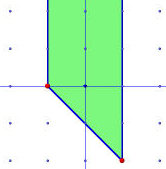}}
\end{center}
\caption{Polar dual of the top in Figure~\ref{F:top}}
\label{F:dualtop}
\end{figure}

\begin{definition}
Let $\Diamond$ be a top.  The \emph{summit} of $\Diamond$ is the intersection of $\Diamond$ with the half-space $x_k \geq 1$.
\end{definition}

\noindent By definition, all of the lattice points of a top lie either in the reflexive boundary or the summit.

Vincent Bouchard and Harald Skarke classified three-dimensional tops in \cite{BS}.  They observed that tops may arise in infinite families, with arbitrarily large numbers of lattice points.  This situation contrasts with the classification of reflexive polytopes: up to changes of coordinates preserving the lattice structure, there are only a finite number of reflexive polytopes in a given dimension.  It follows that there exist tops in every dimension which cannot be completed to reflexive polytopes.  
We illustrate a two-dimensional top that cannot be combined with another top to form a convex reflexive polygon in Figure~\ref{F:nocompletion}.

\begin{figure}[h!]
\begin{center}
\scalebox{.8}{\includegraphics{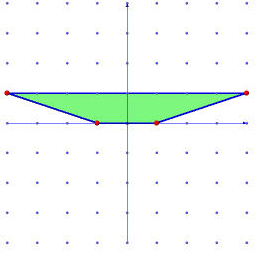}}
\end{center}
\caption{A top that does not complete to a reflexive polytope}
\label{F:nocompletion}
\end{figure}

\subsection{Tops and toric hypersurfaces}

Taking the fan $\Pi$ over the faces of a top defines a toric variety $V_\Pi$; the projection map $(x_1, \dots, x_k) \mapsto x_k$ induces a morphism from $V_\Pi$ to $\mathbb{C}$.  Anticanonical hypersurfaces in $V_\Pi$ are open, $k-1$-dimensional Calabi-Yau varieties; the morphism $V_\Pi \to \mathbb{C}$ induces a map from each of these varieties to $\mathbb{C}$.  Generically, the fiber of this projection map will be a compact $k-2$-dimensional Calabi-Yau variety described by the reflexive boundary of the top.  We may resolve singularities in the generic fiber by choosing a refinement $R$ of $\Pi$ which restricts to a maximal, \textcolor{black}{projective}, simplicial fan on the reflexive boundary; such a fan will include a one-dimensional cone for every nonzero lattice point of the reflexive boundary polytope.  If $k-2 \leq 3$ (so our top is at most five-dimensional), such a fan will yield smooth generic fibers.  We may resolve singularities in our open Calabi-Yau varieties by choosing a maximal \textcolor{black}{projective} simplicial refinement $\Sigma$ of $\Pi$; if $k-1 \leq 3$ (so our top is at most four-dimensional), the open Calabi-Yau varieties will be generically smooth, though the ambient toric variety may have orbifold singularities.  Of course, we can resolve singularities in any dimension if we can find a smooth refinement of $\Pi$.

Alternatively, instead of thinking of smooth anticanonical hypersurfaces as open $k-1$-dimensional Calabi-Yau varieties, we may view them as describing degenerations of $k-2$-dimensional Calabi-Yau varieties.  In this view, working with the refined fan $R$ ensures that the hypersurface describes a degeneration of smooth $k-2$-dimensional Calabi-Yau varieties, while taking $\Sigma$ to be a refinement of $R$ resolves the singularities of the degeneration.  We say a degeneration $X \to U$ is \emph{semistable} if $X$ is non-singular and the fiber $\pi^{-1}(0)$ is reduced, with non-singular components crossing normally.  

We can easily write down the map to $\mathbb{C}$ in homogeneous coordinates.  Suppose $v_1, \dots, v_q$ generate the one-dimensional cones in our fan.  (If we are working with the fan $\Pi$, these generators will just be the vertices of our top; if we are working with $\Sigma$, these will correspond to all of the non-origin lattice points of our top.)  We have corresponding homogeneous coordinates $(z_1, \dots, z_q)$.  Let $h_i$ be the $k$-th coordinate of $v_i$.  Then the map is given by $(z_1, \dots, z_q) \mapsto \prod_{i=1}^q z_i^{h_i}$.  In particular, the preimage of $0 \in \mathbb{C}$ is just given by hyperplanes of the form $z_i=0$, where the corresponding generator lies in the summit of the top.

Bouchard and Skarke studied three-dimensional tops in the context of elliptic fibrations, and described a relationship between points in the summit of a top and twisted Kac-Moody algebras. In many cases, one can read the Dynkin diagram of ADE type which describes the elliptic fibration directly from the summit's points and edges. \cite{BS}
More recently, Candelas, Constantin, and Skarke used four-dimensional tops obtained from slicing reflexive polytopes to describe patterns in the possible Hodge numbers of Calabi-Yau threefolds, \textcolor{black}{and Cicoli, Kreuzer, and Mayrhofer analyzed K3-fibered Calabi-Yau threefolds obtained from slicing four-dimensional reflexive polytopes. \cite{CCS, CKM}}  Grassi and Perduca analyzed a class of reflexive polytopes where both the polytope and its polar dual can be decomposed as a pair of tops.  This construction produces elliptically fibered K3 surfaces which admit semistable degenerations to a pair of rational elliptic surfaces, and can be used to study $F$-theory/Heterotic duality. \cite{GP}

In the current work, we focus on the correspondence between tops and degenerations.  We give a simple combinatorial condition on tops which produces semistable degenerations of K3 surfaces, and (when appropriate smoothness conditions are met) semistable degenerations of Calabi-Yau threefolds.  Our method is constructive: given a fixed reflexive boundary polytope, we apply an algorithm for constructing tops which yields semistable degenerations of the corresponding hypersurfaces.  The properties of each degeneration may be computed directly from the combinatorial structure of the top.  

By a theorem of Mumford, any degeneration may be decomposed as a semistable degeneration followed by a base change. \cite{Mumford}  Thus, the semistable case is the natural starting point for any study of degenerations.  On the other hand, tops provide a concrete, constructive setting for studying more exotic degenerations.  In Proposition~\ref{P:simpTops}, for example, we describe an ``exceptional'' top in every dimension that generalizes the $E_6$ surface singularity.

\section{Short tops}

\begin{definition}
A \emph{short top} is a top where the lattice points in the summit are contained in the hyperplane $x_k = 1$.
\end{definition}

Using polar duality, we see that a top $\Diamond$ is a short top if and only if $\Diamond^\circ$ contains the point $(0,\dots,0,-1)$.  The summit of a short top $\Diamond$ is a facet of the top if and only if the point $(0,\dots,0,-1)$ is a vertex of $\Diamond^\circ$.

\begin{theorem}\label{T:degeneration}
Let $\Diamond$ be a short top, and let $k \leq 5$.  Let $R$ be a maximal simplicial refinement of the fan over the faces of $\Diamond$.  If the nondegenerate anticanonical hypersurface $X_R$ in the toric variety $V_R$ is smooth, then $X_R$ describes a semistable degeneration of smooth $k-2$-dimensional Calabi-Yau varieties.
\end{theorem}

\begin{proof}
The fan conditions ensure that the fiber above a general point $z \in \mathbb{C}$ is a smooth $k-2$-dimensional Calabi-Yau variety.  We need to check that the fiber corresponding to $0 \in \mathbb{C}$ is a reduced divisor and has normal crossings.  Let $S$ be the set of generators of one-dimensional cones of $R$ which lie in the summit of $\Diamond$.  Because $\Diamond$ is a short top, the map $V_R \to \mathbb{C}$ can be written in homogeneous coordinates as
\[(z_1, \dots, z_q) \mapsto \prod_{v_i \in S} z_i. \]  
Let $D_i$ be the toric divisor of $V_R$ given by $z_i = 0$.  Then the divisor of $V_R$ corresponding to $0 \in \mathbb{C}$ is simply $D=\sum_{v_i \in S} D_i$.  This divisor is clearly reduced.  It has normal crossings because it is a sum of toric divisors and $R$ is simplicial; since $X_R$ is nondegenerate, its intersection with $D$ will also be reduced and have normal crossings. 
\end{proof}

\begin{corollary}
Any three-dimensional short top describes a family of semistable degenerations of elliptic curves.  Any four-dimensional short top describes a family of semistable degenerations of K3 surfaces.
\end{corollary}

\begin{corollary}
If $k=5$ and $R$ is a smooth fan, then $X_R$ describes a semistable degeneration of Calabi-Yau threefolds.
\end{corollary}

\begin{remark}
The author of \cite{Hu} gives a construction for semistable degenerations which uses a polytopal decomposition of a \emph{simple} $k$-dimensional polytope in $M_\mathbb{R}$ (dual to a simplicial polytope in $N_\mathbb{R}$) to describe a semistable degeneration of $k-1$-dimensional anticanonical hypersurfaces.
This construction is applied to elliptically fibered K3 surfaces in \cite{GP}: in that setting, decomposing a three-dimensional reflexive polytope in $M_\mathbb{R}$ into two tops glued along a common reflexive boundary yields a semistable degeneration of K3 surfaces to a singular fiber with two components.  In contrast, our construction uses a $k$-dimensional top to describe a semistable degeneration of $k-2$-dimensional varieties.  The higher codimension allows us to work with a larger class of polytopes: in particular, we are able to construct semistable degenerations in the case where the smooth fiber is defined by a reflexive boundary polytope which is not simplicial.  As we discuss below, our construction also allows us to extract information about the singular fiber of a degeneration directly from a polytope, rather than from a polytopal decomposition, which allows for a straightforward analysis of degenerations where the singular fiber has many components.
\end{remark}

\section{Constructing tops}

We wish to construct tops with a given reflexive boundary polytope $\Delta$.  We analyze the equivalent problem of classifying the duals of tops with dual reflexive boundary $\Delta^\circ$.  We know that vertices of the dual top must project to lattice points in the dual of the reflexive boundary.  We have already chosen coordinates for the reflexive boundary polytope.  We may use $GL(k,\mathbb{Z})$ to fix the final coordinates of $k-1$ lattice points of the dual top; we will also determine the final coordinate of a $k$th lattice point based on an analysis of the combinatorial structure of our dual top.  

We wish to choose final coordinates for the remaining lattice points that will yield a \textcolor{black}{dual top.  We need to test two properties: our choices must yield a convex polyhedron, and each facet of the polyhedron must be polar dual to a lattice point in $N$.}  

By \textcolor{black}{\cite[Theorem 6]{Mehlhorn}}, in order to guarantee convexity, it suffices to check \textcolor{black}{a property called \emph{local convexity}.}

\begin{definition}
We say a $k$-dimensional triangulated polytope $\Diamond$ is \emph{locally convex} if for every $k-2$-dimensional face $f$ \textcolor{black}{of the triangulation which lies} in the boundary of $\Diamond$, the simplex defined by the two facets containing $f$ is contained in $\Diamond$.
\end{definition}

Any lattice point triangulation of the finite facets of a dual top will yield a lattice point triangulation of the dual reflexive boundary $\Delta^\circ$ upon vertical projection.  The \emph{regular} triangulations of a polytope are precisely those triangulations which can be obtained by projecting the convex hull of a polytope.  Thus, we may organize our search for dual tops by fixing a regular lattice point triangulation of $\Delta^\circ$ and identifying ways to lift this regular triangulation to a dual top.

One natural way to triangulate a $k-1$-dimensional reflexive polytope is to choose a $k-2$-dimensional triangulation of each facet, and then include the origin as the final vertex of each $k-1$-dimensional simplex.  Because the origin is a vertex of this triangulation, if we lift this triangulation to a lattice triangulation of a dual top, the origin must lift to a lattice point of the dual top.  Because we have chosen the convention that tops lie in the half-space $x_k \geq 0$, this lattice point must be $(0, \dots,0,-1)$.  

On the other hand, any dual of a $k$-dimensional short top may be obtained from a triangulation of the boundary of a $k-1$-dimensional reflexive polytope.  This fact depends on the following lemma, which is proved in \cite{BS} for the case $k=3$:

\begin{lemma}\cite{BS}\label{L:allFacetsHavew}
If $\Diamond^\circ$ is the dual of a short top $\Diamond$, then every bounded facet of $\Diamond^\circ$ contains $(0, \dots,0,-1)$.
\end{lemma}

\begin{proof}
Because $\Diamond$ is a short top, any vertex of $\Diamond$ must have either $x_k=0$ or $x_k=1$.  The vertices of $\Diamond$ are in one-to-one correspondence with the facets of $\Diamond^\circ$.  Vertices at $x_k=0$ define vertical, unbounded facets of $\Diamond^\circ$.  Vertices at $x_k=1$ define facets of $\Diamond^\circ$ containing $(0, \dots,0,-1)$.
\end{proof}

Because every bounded facet of a dual short top contains $(0, \dots,0,-1)$, we may always find a triangulation of the bounded facets of a dual short top that projects to a triangulation of the dual reflexive boundary polytope where every simplex has the origin as a vertex.  We may use this observation to create an algorithm for constructing short tops:

\begin{algorithm}\label{A:shortTopClass}

\textcolor{black}{
\begin{description}
\item[Input]A $(k-1)$-dimensional dual boundary reflexive polytope $\Delta^\circ$ and a regular triangulation $\mathcal{T}$ of the boundary of $\Delta^\circ$.  Let $v_1, \dots, v_q$ be the lattice points of $\Delta^\circ$ which appear as vertices in the triangulation $\mathcal{T}$.  We order the lattice points so that $v_{q-k},\dots,v_{q-1}$ are vertices of a facet of $\mathcal{T}$, and $v_q=(0,\dots,0)$.
\item[Output]A finite list of divisibility conditions and a finite system of linear inequalities on $q-k-1$ integer parameters.  Together, these describe the coordinates of all dual tops corresponding to $\Delta^\circ$ and $\mathcal{T}$, up to overall isomorphism.
\end{description}
}

\textcolor{black}{
\begin{description}
\item[Procedure] \hfill \\
\begin{itemize}
\item Let $a_{v_j}$ be the minimum $y_k$ value of the dual top that projects to $v_j$.
\item Set $a_{(0,\dots,0)}=-1$.
\item Set $a_{v_{q-k}}=-1,\dots,a_{v_{q-1}}=-1$.
\item Now, $a_{v_1},\dots, a_{v_{q-k-1}}$ are the minimum $y_k$ values corresponding to each of the remaining $q-k-1$ lattice points of $\Delta^\circ$.  We will identify values of $a_1,\dots, a_{q-k-1}$ which will result in a dual top.  
\item For each facet $f$ of the regular triangulation $\mathcal{T}$, let $B$ be the $(k-1) \times (k-1)$ matrix where the rows consist of the vertices \textcolor{black}{$v_{j_1}, \dots, v_{j_{k-1}}$} of $f$.  \textcolor{black}{Let $D$ be the Smith normal form of $B$, where $D$ has diagonal entries $d_1, \dots, d_{k-1}$, and write $B = U D V$, where $U$ and $V$ are in $GL_{k-1}(\mathbb{Z})$.  Let $\vec{a}$ be the integer column vector $(a_{v_{j_1}}+1, \dots, a_{v_{j_{k-1}}}+1)^T$.  Return the $k-1$ divisibility conditions $d_i | (U^{-1} \vec{a})_i$.}
\item For each $k-2$-dimensional face $e$ of the regular triangulation $\mathcal{T}$, return a linear inequality in $k+1$ of the $a_{v_j}$ which guarantees that the dual top is locally convex at a face corresponding to $e$.  The $a_{v_j}$ used in this inequality correspond to the $k+1$ lattice points which are vertices of a facet of $\mathcal{T}$ containing $e$.
\end{itemize}
\end{description}
}

\end{algorithm}

\textcolor{black}{Algorithm~\ref{A:shortTopClass} lifts the simplices described by combining a facet of the regular triangulation $\mathcal{T}$ with the origin to facets of a triangulation of the boundary of a dual top.  The divisibility conditions are designed to ensure that the facets of the dual top correspond to points with integer coordinates under polar duality; the system of linear inequalities will ensure that the constructed dual top is convex.}

\textcolor{black}{\begin{lemma}
Let $f$ be a facet of a regular triangulation of the boundary of a $k-1$-dimensional reflexive polytope $\Delta^\circ$, let \textcolor{black}{$v_{j_1} = (y_{11}, \dots, y_{1(k-1)}), \dots, v_{j_{k-1}} = (y_{(k-1)1}, \dots, y_{(k-1)(k-1)})$} be the vertices of $f$, and let \textcolor{black}{$B=UDV$ and $\vec{a}$ be as defined in Algorithm~\ref{A:shortTopClass}}.  The $k-1$-dimensional simplex in $M_\mathbb{R}$ determined by the set $\mathcal{S} = \{(y_{11}, \dots, y_{1(k-1)},a_1), \dots, (y_{(k-1)1}, \dots, y_{(k-1)(k-1)},a_{k-1}), (0, \dots, 0,-1) \}$ corresponds to a lattice point in $N$ under polar duality if and only if \textcolor{black}{$d_i | (U^{-1} \vec{a})_i$.}
\end{lemma}}

\textcolor{black}{\begin{proof}
We wish to find the equation of the hyperplane determined by $\mathcal{S}$.  We may do so by finding $\vec{x} \in N_\mathbb{R}$ such that $\vec{x} \cdot (v_j-(0, \dots, 0,-1))=0$.  This is equivalent to finding the null space of the augmented matrix 
\[A = \left[\begin{BMAT}(@){c|c}{ccc}
& a_1+1 \\
B & \vdots \\
 & a_{k-1} +1
  \end{BMAT}\right].\]
Applying Gauss-Jordan elimination, we find that the null space is generated by the vector
\[\vec{x} = \left[\begin{BMAT}(@){c}{c|c} B^{-1} \begin{pmatrix}a_1+1 \\ \vdots \\ a_{k-1}+1 \end{pmatrix} \\ 1 \end{BMAT}\right].\]
The equation of our hyperplane is given by $\vec{x} \cdot \vec{y} = \vec{x} \cdot (0, \dots, 0,-1)$, or equivalently $\vec{x} \cdot \vec{y} = -1$.  Thus, $\vec{x}$ is precisely the point in $N_\mathbb{R}$ dual to our hyperplane, and we see that $\vec{x}$ will lie in $N$ if and only if $B^{-1} \begin{pmatrix}a_1+1 \\ \vdots \\ a_{k-1}+1\end{pmatrix}$ has integer entries.  But this will hold if and only if \textcolor{black}{$d_i | (U^{-1} \vec{a})_i$, because $U$ and $V$ are in $GL_{k-1}(\mathbb{Z})$}.
\end{proof}}

\textcolor{black}{We now describe the process of checking local convexity.  Let $e$ be a $k-2$-dimensional face of the triangulation $\mathcal{T}$, and let $v_1, \dots, v_{k-2}$ be the vertices of $e$.  The face $e$ is contained in precisely two facets $f_1$ and $f_2$ of $\mathcal{T}$; let $v_{k-1}$ and $v_{k}$ be the remaining vertices of these facets.  Checking local convexity at the lift of $e$ involves checking whether the simplex defined by the lifts of $f_1$ and $f_2$ is contained in our candidate dual top.  We may test containment by testing the orientation of the simplex.  In turn, we compute the orientation of this simplex by checking the sign of the determinant
\[\begin{vmatrix}v_{11} & \dots & v_{1(k-1)}& a_{v_1}& 1\\
\vdots & \dots & & \vdots & 1 \\
v_{k1} & \dots & v_{k(k-1)}& a_{v_k}& 1\end{vmatrix}.\]
By expanding along the $k$th column, we see that this determinant corresponds to a linear condition in the integer parameters $a_{v_j}$.}

\textcolor{black}{\begin{theorem}For a fixed reflexive boundary polytope $\Delta^\circ$ and regular triangulation $\mathcal{T}$ of the boundary of $\Delta^\circ$, there exist infinitely many choices of the parameters $a_{v_1},\dots, a_{v_{q-k-1}}$ which satisfy the divisibility conditions and system of linear inequalities produced by Algorithm~\ref{A:shortTopClass}.
\end{theorem}}

\textcolor{black}{\begin{proof}
Because we have a finite number of divisibility conditions, there are an infinite number of choices for each integer $a_{v_j}$ satisfying all of the divisibility conditions.  In particular, note that $a_{v_j}=-1$ will always satisfy all of the divisibility conditions.  Now, suppose we set $a_{v_j}=-1$ for all $j>1$.  In this setting, the only nontrivial convexity conditions are given by the sign of a determinant of the form
\[\begin{vmatrix}v_{11} & \dots & v_{1(k-1)}& a_{v_1}& 1\\
w_{21} & \dots & w_{2(k-1)}& -1 & 1\\
\vdots & \dots & & \vdots & 1 \\
w_{k1} & \dots & w_{k(k-1)}& -1 & 1\end{vmatrix}.\]
These determinants will change sign when $a_{v_1}=-1$; because we have chosen the convention that our dual tops extend in the positive $x_k$ direction, we may reduce to the single inequality $a_{v_1}\geq-1$.
\end{proof}}

To classify all polar duals of short tops arising from a fixed dual reflexive boundary, we must enumerate all regular triangulations of the boundary and then eliminate the duplicate dual tops which arise from more than one triangulation.  For example, any regular triangulation of the boundary will produce the dual top which has all vertices at $y_k=-1$.  In general, enumerating the regular triangulations of the boundary of a reflexive polytope is highly computationally complex.  However, for many well-known families of reflexive polytopes, describing regular triangulations is much simpler.  As an example, we classify all $k$-dimensional tops with the standard $(k-1)$-dimensional simplex as dual reflexive boundary.

\begin{proposition}[The standard simplex]\label{P:simpTops}

Let $\Delta^\circ$ be the $(k-1)$-dimensional simplex with vertices at $(1,0,\dots, 0), \dots, (0,\dots, 0, 1)$, and $(-1,\dots, -1)$, and let $\Delta$ be the polar dual of $\Delta^\circ$.  Then any top with base $\Delta$ is equivalent to one of the following:

\begin{enumerate}
\item A member of the one-parameter family of short tops with summit vertices given by

\begin{align*}
(0, &\dots, 0, 1)\\
(a+1, 0, &\dots, 0, 1)\\
(0, a+1, &\dots, 0, 1)\\
&\vdots\\
(0, 0, \dots&, a+1, 1)
\end{align*}

\noindent where $a \geq -1$ is an integer.
\item The top with a single summit vertex given by $(-1,-1,\dots,-1,k)$.
\end{enumerate}
\end{proposition}

\begin{proof}
There is only one lattice point triangulation of the boundary of $\Delta^\circ$.  Applying Algorithm~\ref{A:shortTopClass}, we obtain the family of short tops listed above.

On the other hand, there is only one lattice point triangulation of $\Delta^\circ$ which does not use the origin, namely, the triangulation consisting of a single simplex.  There are $k$ vertices of $\Delta^\circ$ and we have $k$ degrees of freedom, so any dual top which does not have the point below the origin as a lattice point must be equivalent to any other dual point with the same property.  We choose the dual top with the following vertices as our canonical form:

\begin{align*}
(1, &\dots, 0, 0)\\
(0, 1, &\dots, 0, 0)\\
&\vdots\\
(0, 0, &\dots, 1, 0)\\
(-1, -1, \dots&, -1, -1)
\end{align*}

This dual top has a single bounded facet, which corresponds under polar duality to the vertex $(-1,-1,\dots,-1,k)$ of the top.

\end{proof}

\begin{remark}
In three dimensions, the second case in Proposition~\ref{P:simpTops} corresponds to an $E_6$ singularity. \cite{BS} 
\end{remark}

Because every facet of a $k$-dimensional smooth Fano polytope contains precisely $k$ lattice points, the boundary of any smooth Fano polytope admits a unique lattice triangulation.  
We implemented a procedure in Sage to compute all isomorphism classes of short tops with dual reflexive base $\Delta^\circ$ a 3-dimensional smooth Fano polytope.\cite{Sage}  For a smooth Fano polytope with $m$ vertices, we obtain an infinite family of short tops with $m-3$ parameters.  \textcolor{black}{(Note that the fact that smooth Fano polytopes are smooth ensures that the divisibility conditions on our parameters will be trivial.)}

\begin{example}
Let $\Delta^\circ$ be the smooth Fano polytope with vertices at $\left(1, 0, 0\right)$, $\left( 0, 1, 0 \right)$, $\left( 0, 0, 1\right)$, $\left( -1, 0, 0\right)$, and $\left( 0, -1, -1\right)$.  We illustrate $\Delta^\circ$ and its polar dual $\Delta$ in Figures~\ref{F:fano} and \ref{F:fanodual}.
%
%\begin{multicols}{2}

\begin{figurehere}
\scalebox{.3}{\includegraphics{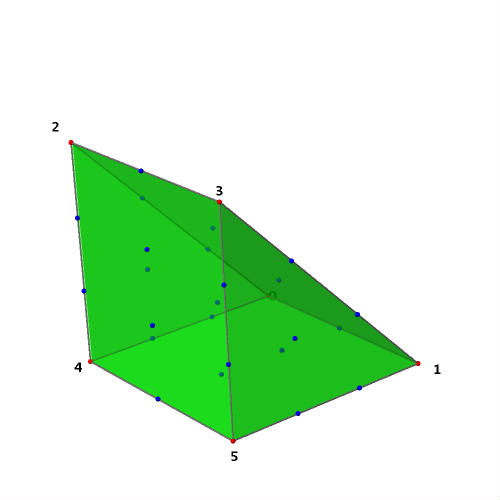}}
\caption{Polar dual polytope $\Delta$}
\label{F:fanodual}
\end{figurehere}

\begin{figurehere}
\scalebox{.3}{\includegraphics{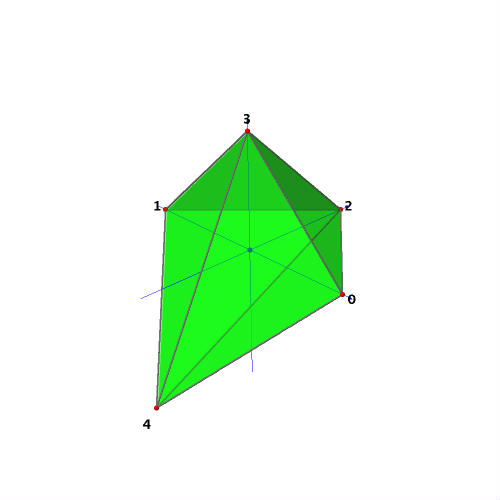}}
\caption{The Fano polytope $\Delta^\circ$}
\label{F:fano}
\end{figurehere}

%\end{multicols}

The dual top $\Diamond^\circ$ has two free integer parameters, $a_1$ and $a_2$.  To make $\Diamond^\circ$ convex, we require $a_1 \geq -1$ and $a_2 \geq -1$.  We illustrate the summits of the resulting tops for the parameter choices $a_1=0$, $a_2=4$ and $a_1=4$, $a_2=0$ in Figures~\ref{F:summit1} and \ref{F:summit2}.

%\begin{multicols}{2}
\begin{figurehere}
\scalebox{.3}{\includegraphics{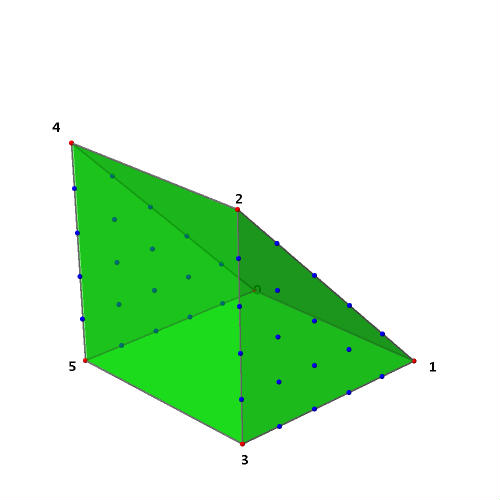}}
\caption{Summit of the short top for $a_1=0$ and $a_2=4$}
\label{F:summit1}
\end{figurehere}

\begin{figurehere}
\scalebox{.3}{\includegraphics{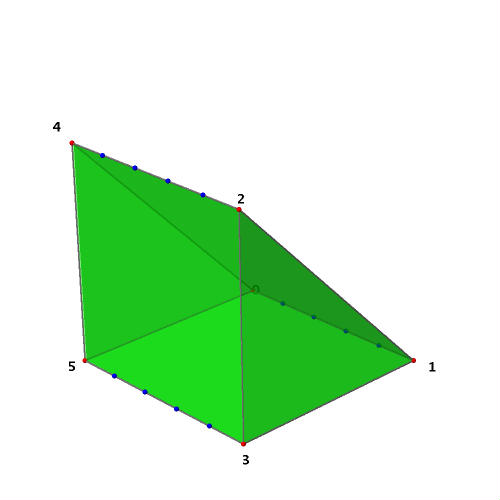}}
\caption{Summit of the short top for $a_1=4$ and $a_2=0$}
\label{F:summit2}
\end{figurehere}

%\end{multicols}

\end{example}
%
%We describe all of the short tops obtained from three-dimensional smooth Fano polytopes in Appendix~\ref{App:fano}.

\section{Semistable degenerations of K3 surfaces}

Let $\Sigma$ be a maximal simplicial fan which refines $R$.  Then the map $X_\Sigma \to X_R$ resolves the singularities of the degeneration.  Because $\Sigma$ is maximal, every lattice point in the summit of $\Diamond$ determines a toric divisor in the preimage of $0 \in \mathbb{C}$.  Divisors corresponding to lattice points strictly in the interior of facets of $\Diamond$ will not intersect $X_\Sigma$.  

We say a toric divisor \emph{splits} if its intersection with $X_\Sigma$ has more than one component.  Whether a toric divisor splits is determined by the structure of $\Diamond$ and $\Diamond^\circ$: the divisor corresponding to a lattice point splits if that lattice point lies in the relative interior of a $k-2$-face $\theta$ of $\Diamond$, and the dual face $\theta^\circ$ also contains lattice points in its interior.  Note that $\theta^\circ$ will be a one-face of $\Diamond^\circ$, also known as an edge.  When $\Diamond$ is a short top, the dual to any face $\theta$ intersecting the summit of $\Diamond$ will contain the point $(0,\dots,0,-1)$.  In order to have splitting, $(0,\dots,0,-1)$ must lie in the interior of an edge of $\Diamond^\circ$.  Because lattice points of the dual top project vertically to lattice points of the dual reflexive boundary, such an edge will have precisely 3 lattice points.  It follows from standard results on Calabi-Yau hypersurfaces in toric varieties (cf.~\cite{CoxKatz}) that when a toric divisor in a short top splits, it will yield precisely 2 components.

Now, let $\Diamond$ be a four-dimensional short top.  The nondegenerate anticanonical hypersurfaces $X_R$ describe semistable degenerations of K3 surfaces by Theorem~\ref{T:degeneration}.  We can classify four-dimensional short tops based on the position of the point $w = (0,\dots,0,-1)$ in the dual top.  We have the following cases:

\begin{enumerate}
\item The point $w$ lies in the interior of a facet of $\Diamond^\circ$.
\item The point $w$ lies in the interior of a two-face of $\Diamond^\circ$.

\item \begin{enumerate}
\item The point $w$ lies in the interior of an edge of $\Diamond^\circ$.
\item The point $w$ is a vertex of $\Diamond^\circ$.
\end{enumerate}

\end{enumerate}

Under polar duality, these cases correspond to the following descriptions of the summit of the top:

\begin{enumerate}
\item The summit of $\Diamond$ consists of a single lattice point.
\item The summit is an edge of $\Diamond$.

\item \begin{enumerate}
\item The summit is a two-face of $\Diamond$.
\item The summit is a facet of $\Diamond$, and therefore a three-dimensional lattice polytope.
\end{enumerate}

\end{enumerate}

We wish to understand the relationship between the combinatorial structure of the summit of the short top and the semistable degeneration of K3 surfaces described by the short top.  Semistable degenerations of K3 surfaces have been classified:

\begin{theorem}\cite{Kulikov,Persson,FM}\label{T:semistableK3}
Let $\pi: X \to D$ be a semistable degeneration of K3 surfaces with trivial canonical bundle $\omega_{X} \cong \mathcal{O}_{X}$.  Let $X_0 = \pi^{-1}(0)$, and assume all components of $X_0$ are K\"{a}hler.  Then either:

\begin{enumerate}
\item $X_0$ is a smooth K3 surface.
\item $X_0$ is a chain of elliptic ruled components with rational surfaces at
each end.
\item $X_0$ consists of rational surfaces meeting along rational curves.  The dual graph of $X_0$ has the sphere as topological support.
\end{enumerate}
\end{theorem}

Semistable degenerations of K3 surfaces determined by tops satisfy the hypotheses of Theorem~\ref{T:semistableK3}.  Combining the combinatorial data encoded by a short top with the results of Theorem~\ref{T:semistableK3} yields the following classification:

\begin{proposition}\label{P:shortTopK3Degenerations}

Let $\Diamond$ be a four-dimensional short top, and let $\mathcal{T}$ be a triangulation of the boundary of $\Diamond$ induced by a maximal simplicial fan $\Sigma$.  Then $\Diamond$ determines a semistable degeneration of K3 surfaces that falls into one of the following cases.

\begin{enumerate}
\item If the summit of $\Diamond$ consists of a single lattice point, then $X_0$ is a smooth K3 surface.
\item If the summit is an edge of $\Diamond$, then $X_0$ is a chain of elliptic ruled components with rational surfaces at
each end.  Each component corresponds to a lattice point in the summit; in particular, the two vertices of the summit correspond to the two rational surfaces at the ends of the chain.

\item \begin{enumerate}
\item If the summit is a two-face $F$ of $\Diamond$, then $X_0$ consists of rational surfaces meeting along rational curves.  The lattice points on the relative boundary of $F$ correspond to a single component of $X_0$.  Each lattice point in the relative interior of $F$ corresponds to 2 components of $X_0$. The dual graph $\Gamma$ of $X_0$ has the sphere as topological support. There is one edge $\Gamma$ for each edge in $\mathcal{T}$ connecting points on the relative boundary of $F$, and there are two edges in $\Gamma$ for each edge in the triangulation of $F$ induced by $\mathcal{T}$ which has an endpoint in the relative interior of $F$.  
\item If the summit is a three-dimensional lattice polytope $P$, then $X_0$ consists of rational surfaces meeting along rational curves.  The vertices of the dual graph $\Gamma$ of $X_0$ are in one-to-one correspondence with the lattice points of the boundary of $P$, and the edges of $\Gamma$ are given by the triangulation of the boundary of $P$ induced by $\mathcal{T}$.
\end{enumerate}

\end{enumerate}
\end{proposition}

\begin{proof}
If the summit of $\Diamond$ consists of a single lattice point, then $X_0$ has a single component, and we are in Case 1 of Theorem~\ref{T:semistableK3}.

If the summit is an edge of $\Diamond$, then each lattice point in the summit corresponds to a single component of $X_0$, because the toric divisors described by the lattice points do not split.  Divisors have nontrivial intersection if and only if the corresponding lattice points are connected by an edge in $\mathcal{T}$.  It follows that $X_0$ is a chain of surfaces, so we are in Case 2 of Theorem~\ref{T:semistableK3}.

If the summit is a two-face $F$ of $\Diamond$, then the point $w = (0,\dots,0,-1)$ lies in the interior of an edge of $\Diamond^\circ$.  In this case, toric divisors corresponding to lattice points in the relative interior of $F$ will split into 2 components, thereby yielding 2 components of $X_0$.  The toric divisors corresponding to lattice points on the relative boundary of $F$ will not split.  An edge in $\mathcal{T}$ that connects two lattice points on the relative boundary of $F$ yields an edge in the dual graph $\Gamma$ of $X_0$, because the corresponding pairs of divisors have non-trivial intersection.  We may analyze intersections for split divisors following the argument in \cite{Rohsiepe}.  We find that an edge in $\mathcal{T}$ between a lattice point on the boundary of $F$ and an interior lattice point of $F$ will yield 2 edges in $\Gamma$, one for each component of $X_0$ obtained from the interior lattice point.  An edge in $\mathcal{T}$ connecting 2 interior lattice points of $F$ will also yield 2 edges in $\Gamma$, one connecting the first component obtained from each lattice point and one connecting the second component obtained from each lattice point.  Because $X_0$ consists of neither a single component nor a chain of surfaces, we are in Case 3 of Theorem~\ref{T:semistableK3}.

If the summit is a three-dimensional lattice polytope $P$, then splitting cannot occur.  Every lattice point on the relative boundary of $P$ will yield a component of $X_0$; the lattice points in the relative interior of $P$ correspond to toric divisors that do not intersect $X_0$.  The edges of $\Gamma$ are given by the triangulation of the boundary of $P$ induced by $\mathcal{T}$.  Because $X_0$ consists of neither a single component nor a chain of surfaces, we are in Case 3 of Theorem~\ref{T:semistableK3}.

\end{proof}

Given a fixed three-dimensional reflexive boundary polytope $\Delta$, we can always construct a short top falling into Case 1 of Proposition~\ref{P:shortTopK3Degenerations} by adding a single summit point at $(0,0,0,1)$, and we can always construct short tops falling into Case 3(b) by choosing a triangulation of the facets of the dual reflexive boundary polytope $\Delta^\circ$ and applying Algorithm~\ref{A:shortTopClass}.  Whether Cases 2 and 3(a) arise depends on the combinatorial structure of $\Delta^\circ$.

\begin{lemma}\label{L:case2}
A three-dimensional reflexive polytope $\Delta$ can be the boundary of a short top $\Diamond$ falling into Case 2 of Proposition~\ref{P:shortTopK3Degenerations} if and only if the polar dual $\Delta^\circ$ is isomorphic to a pair of three-dimensional tops glued along a common two-dimensional reflexive boundary polytope.
\end{lemma}

\begin{proof}
The short top $\Diamond$ falls into Case 2 of Proposition~\ref{P:shortTopK3Degenerations} if the point $w = (0, \dots,0,-1)$ lies in the interior of a two-face $F$ of $\Diamond^\circ$.  The vertical projection of $F$ onto $\Delta^\circ$ is the intersection of a hyperplane $H$ containing the origin with $\Delta^\circ$.  The lattice polygon $H \cap \Delta^\circ$ is a two-dimensional reflexive polygon which divides $\Delta^\circ$ into two pieces $A$ and $B$, each isomorphic to a top.  Conversely, given such a subdivision of $\Delta^\circ$, we can construct a dual short top $\Diamond^\circ$ corresponding to Case 2 by choosing a lattice point triangulation of $\Delta^\circ$ consistent with the subdivision, assigning a minimum $x_k$ value of $-1$ to all of the lattice points in $A$, and choosing $x_k$ values for points in $B$ but not $A$ that are at least $0$ and are consistent with convexity.
\end{proof}

Ten of the 4,319 classes of three-dimensional reflexive polytopes do not satisfy the condition given in Lemma~\ref{L:case2}.  In the database given in Sage, the dual polytopes $\Delta^\circ$ have indices 0 (the standard simplex), 2, 5, 7, 16, 26, 31, 37, 40, and 53. \cite{Sage}  \textcolor{black}{These polytopes cannot be used to construct degenerations of K3 surfaces where $X_0$ is a chain of elliptic ruled components.}    

\begin{lemma}\label{L:case3aIFF}
A three-dimensional reflexive polytope $\Delta$ can be the boundary of a short top $\Diamond$ falling into Case 3(a) of Proposition~\ref{P:shortTopK3Degenerations} if and only if the polar dual $\Delta^\circ$ admits a polytopal decomposition such that the vertices of each polytope in the decomposition are lattice points and each three-dimensional polytope $\delta$ in the decomposition contains a fixed edge $e$ with the origin in its interior.
\end{lemma}

\begin{proof}
The short top $\Diamond$ falls into Case 3(a) of Proposition~\ref{P:shortTopK3Degenerations} if the point $w = (0, \dots,0,-1)$ lies in the interior of an edge $E$ of $\Diamond^\circ$.  Every bounded facet of $\Diamond^\circ$ contains $E$ by Lemma~\ref{L:allFacetsHavew}.  The vertical projection of the bounded facets yields the desired polytopal decomposition.  Conversely, given such a decomposition of $\Delta^\circ$, we can construct a dual short top $\Diamond^\circ$ corresponding to Case 3(a) by assigning a minimum $x_k$ value of $-1$ to the lattice points in $e$, and choosing $x_k$ values for the other lattice points that lift the polytopal decomposition to facets and are consistent with convexity.
\end{proof}

Enumerating the polytopal decompositions of a given lattice polytope is a computationally complex task.  \textcolor{black}{We may} describe necessary and sufficient conditions for Case 3(a) which are less computationally intensive to check \textcolor{black}{by focusing on triangulations rather than arbitrary polytopal decompositions}.

\begin{lemma}\label{L:case3a}
If a three-dimensional reflexive polytope $\Delta$ can be the boundary of a short top $\Diamond$ falling into Case 3(a) of Proposition~\ref{P:shortTopK3Degenerations}, then the origin lies in the interior of a line segment $e$ between two lattice points of $\Delta^\circ$.  If $\Delta^\circ$ admits a lattice point triangulation such that every three-dimensional simplex in the triangulation contains a line segment $e$ between two lattice points on the boundary of $\Delta^\circ$ with the origin as an interior point, then we can construct a short top $\Diamond$ falling into Case 3(a) of Proposition~\ref{P:shortTopK3Degenerations} with $\Delta$ as reflexive boundary.
\end{lemma}

\begin{proof}
A short top $\Diamond$ falls into Case 3(a) of Proposition~\ref{P:shortTopK3Degenerations} if the point $w = (0, \dots,0,-1)$ lies in the interior of an edge of $\Diamond^\circ$.  The vertical projection of this edge onto $\Delta^\circ$ is a line segment between two lattice points on the boundary of $\Delta^\circ$.  On the other hand, given a lattice point triangulation such that every three-dimensional simplex in the triangulation contains a line segment $e$ between two lattice points on the boundary of $\Delta^\circ$ with the origin as an interior point, we can construct a dual short top $\Diamond^\circ$ where the finite facets of the dual top are in one-to-one correspondence with the three-dimensional simplices in our triangulation by applying Algorithm~\ref{A:shortTopClass}.  Such a dual top will correspond to Case 3(a).
\end{proof}

Thirteen of the 4,319 classes of three-dimensional reflexive polytopes do not satisfy the necessary condition of Lemma~\ref{L:case3a}.  In the database given in Sage, the dual polytopes $\Delta^\circ$ where the origin does not lie in the interior of a line segment between two lattice points have indices 0, 1, 3, 6, 13, 22, 33, 54, 68, 87, 90, 98, and 118. \cite{Sage}

As an example of the application of Proposition~\ref{P:shortTopK3Degenerations}, we classify all short tops where the dual reflexive boundary $\Delta^\circ$ is the octahedron with vertices $(\pm 1,0,0)$, $(0,\pm 1,0)$, and $(0,0,\pm 1)$.

\begin{example}
Let $\Delta$ be the cube with vertices of the form $(\pm 1,\pm 1,\pm 1)$.  Then any short top with reflexive boundary $\Delta$ is isomorphic to a short top $\Diamond$ which falls into one of the following cases.
\begin{enumerate}
\item The summit of $\Diamond$ consists of the point $(0,0,0,1)$.
\item The summit is an edge with vertices $(0,0,0,1)$ and $(0,0,c+1,1)$, where $c$ is an integer and $c>-1$.
\item \begin{enumerate}
\item The summit is a quadrilateral with vertices $(0,0,0,1)$, $(0,b+1,0,1)$, $(0,0,c+1,1)$, and $(0,b+1,c+1,1)$, where the integer parameters $b$ and $c$ satisfy $b>-1$, $c>-1$.
\item The summit is a three-dimensional rectangular parallelepiped with vertices $(0,0,0,1)$, $(a+1,0,0,1)$, $(0,b+1,0,1)$, $(0,0,c+1,1)$, $(a+1,b+1,0,1)$, $(a+1,0,c+1,1)$, $(0,b+1,c+1,1)$, and $(a+1,b+1,c+1,1)$, where the integer parameters $b$ and $c$ satisfy $a>-1$, $b>-1$, and $c>-1$.
\end{enumerate}
\end{enumerate}
\end{example}

\begin{proof}
If the dual top $\Diamond^\circ$ has a single finite facet at $x_4=-1$, we obtain a summit of $\Diamond$ with the single point $(0,0,0,1)$.  Because $\Delta^\circ$ is a smooth Fano polytope, there is a unique lattice point triangulation of the facets of $\Delta^\circ$, so we may obtain all smooth tops falling into Case 3(b) of Proposition~\ref{P:shortTopK3Degenerations} (up to isomorphism) by applying Algorithm~\ref{A:shortTopClass}.

We may divide $\Delta^\circ$ into two pieces isomorphic to tops by splitting it into the points satisfying $x_3 \geq 0$ and the points satisfying $x_3 \leq 0$.  Up to isomorphism, this is the unique such division.  We apply Lemma~\ref{L:case2} to obtain a family of smooth tops satisfying Case 2 of Proposition~\ref{P:shortTopK3Degenerations}.

There is a lattice point triangulation of $\Delta^\circ$ that consists of four simplices with the edge between $(1,0,0)$ and $(-1,0,0)$ as common intersection.  Up to isomorphism, this is the unique lattice point triangulation of $\Delta^\circ$ where the origin is interior to an edge of every simplex.  We apply Lemma~\ref{L:case3a} to obtain a family of smooth tops satisfying Case 3(a) of Proposition~\ref{P:shortTopK3Degenerations}.
\end{proof}

\section{Semistable degenerations of Calabi-Yau threefolds}

When a five-dimensional short top defines a semistable degeneration of Calabi-Yau threefolds, we may use the summit of the short top to give a combinatorial description of the degeneration, in analogy to Proposition~\ref{P:shortTopK3Degenerations}.

\begin{proposition}\label{P:shortTopCY3Degenerations}

Let $\Diamond$ be a five-dimensional short top, and let $\mathcal{T}$ be a triangulation of the boundary of $\Diamond$ induced by a maximal simplicial fan $\Sigma$.  Suppose $X_\Sigma$ is smooth.
Then $\Diamond$ determines a semistable degeneration of Calabi-Yau threefolds that falls into one of the following cases.

\begin{enumerate}
\item If the summit of $\Diamond$ consists of a single lattice point, then $X_0$ is a smooth Calabi-Yau threefold.
\item If the summit is an edge of $\Diamond$ containing $\ell$ lattice points, then $X_0$ is a chain of $\ell$ components.
\item If the summit is a two-face $F$ of $\Diamond$, then the dual graph $\Gamma$ of $X_0$ has a disk as topological support.  The vertices of the dual graph $\Gamma$ of $X_0$ are in one-to-one correspondence with the lattice points of $F$, and the edges of $\Gamma$ are given by the triangulation of the boundary of $F$ induced by $\mathcal{T}$.
\item \begin{enumerate}
\item If the summit is a three-face $G$ of $\Diamond$, then the lattice points on the relative boundary of $G$ correspond to a single component of $X_0$.  Each lattice point in the relative interior of $G$ corresponds to 2 components of $X_0$. There is one edge in the dual graph $\Gamma$ of $X_0$ for each edge in $\mathcal{T}$ connecting points on the relative boundary of $F$, and there are two edges in $\Gamma$ for each edge in the triangulation of $F$ induced by $\mathcal{T}$ which has an endpoint in the relative interior of $F$.  
\item If the summit is a four-dimensional lattice polytope $P$, then the dual graph $\Gamma$ of $X_0$ has the three-sphere as topological support.  The vertices of $\Gamma$ are in one-to-one correspondence with the lattice points of the boundary of $P$, and the edges of $\Gamma$ are given by the triangulation of the boundary of $P$ induced by $\mathcal{T}$.
\end{enumerate}

\end{enumerate}
\end{proposition}

We do not have a complete classification of semistable degenerations of Calabi-Yau threefolds analogous to Theorem~\ref{T:semistableK3}.  However, we may obtain a rough classification by analyzing the monodromy of the degeneration.  Let $\pi: X \to D$ be a degeneration, and let $X_t$ be a fixed smooth fiber.  The restriction of $\pi$ to $D-\{0\}$ induces an action of the fundamental group $\pi_1(D-\{0\}) \cong \mathbb{Z}$ on the cohomology groups $H^m(X_t)$.  The \emph{Picard-Lefshetz transformation} is the map $T:H^m(X_t) \to H^m(X_t)$ induced by the canonical generator of $\pi_1(D-\{0\})$.

\begin{theorem}[Monodromy Theorem]\cite{Landman}
If $\pi: X \to D$ is a semistable degeneration, then $T$ is unipotent, with index of unipotency at most $m$.  Thus, $(T-I)^{m+1}=0$, where $I$ is the identity.
\end{theorem}

We may define a nilpotent operator $N$ by the finite sum
\[N = \log T = (T-I) - \frac{1}{2} (T-I)^2 + \frac{1}{3} (T-I)^2 - \dots.\]
The index of unipotency of $T$ is the same as the index of nilpotency of $N$.  One may use $N$ to define an exact sequence known as the Clemens-Schmid exact sequence; an expository treatment may be found in \cite{MorrisonCS}.

\begin{lemma}\cite{MorrisonCS}\label{L:dualGraphNilpotency}
Let $\pi: X \to D$ be a semistable degeneration, and let $\Gamma$ be the dual graph of the singular fiber.  Then $N^{m+1}: H^m(X_t) \to H^m(X_t)$ is always 0, and $N^{m}: H^m(X_t) \to H^m(X_t)$ is 0 if and only if $H^m(|\Gamma|)=0$.
\end{lemma}

By combining Lemma~\ref{L:dualGraphNilpotency} with the classification in Proposition~\ref{P:shortTopCY3Degenerations}, we may characterize semistable degenerations obtained from short tops that yield the maximum index of nilpotency:

\begin{corollary}
Let $\Diamond$ be a five-dimensional short top that determines a semistable degeneration of Calabi-Yau threefolds.  Then $N^3: H^*(X_t) \to H^*(X_t)$ is nontrivial if and only if the summit of $\Diamond$ is a three-face or a four-dimensional lattice polytope.
\end{corollary}

\textcolor{black}{One is naturally led to ask whether the types of degenerations described in Proposition~\ref{P:shortTopCY3Degenerations} are the only possible semistable degenerations of Calabi-Yau threefolds:}

\textcolor{black}{
\begin{question}If $\pi: X \to D$ is a semistable degeneration of Calabi-Yau threefolds, is $X_0$ always either a smooth Calabi-Yau threefold, a chain of $\ell$ components, or described by a dual graph that has a disk or the three-sphere as topological support?
\end{question}
}

\bibliographystyle{alpha}
\bibliography{topsbibarxiv}

%\appendix
%\include{fanoappendix}

\end{document}